\documentclass{amsart}

\usepackage{amssymb, amscd, amsthm, mathrsfs, bm}

\usepackage{color}
\usepackage[colorlinks=true]{hyperref}

\newtheorem{Thm}{Theorem}[section]
\newtheorem{Prop}[Thm]{Proposition}

\newtheorem{Lem}[Thm]{Lemma}

\newtheorem{Conj}[Thm]{Conjecture}
\theoremstyle{remark}

\newtheorem*{Notation}{Notation}
\newtheorem*{Ack}{Acknowledgments}

\numberwithin{equation}{section}

\newcommand{\Order}{\mathcal{O}}
\newcommand{\into}{\hookrightarrow}
\newcommand{\intoback}{\hookleftarrow}
\newcommand{\onto}{\twoheadrightarrow}
\newcommand{\isomto}{\overset{\sim}{\to}}

\newcommand{\compose}{\mathbin{\circ}}
\newcommand{\tensor}{\mathbin{\otimes}}

\newcommand{\Z}{\mathbb{Z}}
\newcommand{\Q}{\mathbb{Q}}

\newcommand{\F}{\mathbb{F}}
\newcommand{\et}{\mathrm{et}}

\newcommand{\fppf}{\mathrm{fppf}}

\newcommand{\id}{\mathrm{id}}
\newcommand{\Gm}{\mathbf{G}_{m}}
\newcommand{\Ga}{\mathbf{G}_{a}}

\newcommand{\invlim}{\varprojlim}

\mathchardef\mhyphen="2D

\newcommand{\alg}[1]{\mathbf{#1}}
\newcommand{\var}{\;\cdot\;}

\newcommand{\ideal}[1]{\mathfrak{#1}}

\newcommand{\Frob}{\mathrm{Fr}}

\DeclareMathOperator{\Ker}{Ker}

\DeclareMathOperator{\Spec}{Spec}
\DeclareMathOperator{\Ab}{Ab}

\DeclareMathOperator{\Pic}{Pic}

\DeclareMathOperator{\Lie}{Lie}

\let\Im\relax

\DeclareMathOperator{\Im}{Im}

\DeclareMathOperator{\gr}{gr}

\title[Finite generation for punctured spectra]
	{Finite generation of nilpotent quotients of fundamental groups of punctured spectra}
\author{Takashi Suzuki}
\address{
	Department of Mathematics, Chuo University,
	1-13-27 Kasuga, Bunkyo-ku, Tokyo 112-8551, Japan
}
\email{tsuzuki@gug.math.chuo-u.ac.jp}
\date{July 4, 2023}
\subjclass[2010]{Primary: 14B05; Secondary: 14F30, 11G25, 14G20, 14J17}
\keywords{Local fundamental groups; singularities; $p$-adic nearby cycles}

\begin{document}

\begin{abstract}
	In SGA 2, Grothendieck conjectures that
	the \'etale fundamental group of the punctured spectrum of a complete noetherian local domain
	of dimension at least two
	with algebraically closed residue field is topologically finitely generated.
	In this paper, we prove a weaker statement,
	namely that the maximal pro-nilpotent quotient of the fundamental group is topologically finitely generated.
	The proof uses $p$-adic nearby cycles and negative definiteness of intersection pairings
	over resolutions of singularities
	as well as some analysis of Lie algebras of certain algebraic group structures on deformation cohomology.
\end{abstract}

\maketitle

\tableofcontents

%%%%%%%%%%%%%%%%%%%%%%%%%%%%%%%%%%%%%%%%%%%%%%%%%%%%%%%%%%%%%%%%%%%%%%%%%%%%%

\section{Introduction}
\label{0043}

In SGA 2 \cite{Gro05}, Grothendieck makes the following conjecture:

\begin{Conj}[{\cite[Expos\'e XIII, Conjecture 3.1]{Gro05}}] \label{0017}
	Let $A$ be a complete noetherian local ring with algebraically closed residue field $F$
	and maximal ideal $\ideal{m}$.
	Let $p$ be the characteristic of $F$ if it is positive
	and let $p = 1$ otherwise.
	Assume that the irreducible components of $\Spec A$ have dimension $\ge 2$
	and the scheme $\Spec A \setminus \{\ideal{m}\}$ is connected.
	Then:
	\begin{enumerate}
		\item \label{0041}
			The \'etale fundamental group $\pi_{1}(\Spec A \setminus \{\ideal{m}\})$ is
			topologically finitely generated.
		\item \label{0042}
			The maximal pro-prime-to-$p$ quotient of $\pi_{1}(\Spec A \setminus \{\ideal{m}\})$
			is topologically finitely presented.
	\end{enumerate}
\end{Conj}

This conjecture is in part based on
Mumford's earlier study \cite{Mum61} in the complex-analytic setting,
where a topological analogue of $\pi_{1}(\Spec A \setminus \{\ideal{m}\})$
is shown to be finitely presented
when the exceptional divisor of a resolution of singularities of $A$ is simply connected
(see \cite[Section I, pp.11--12]{Mum61}).
Statement \eqref{0042} is proved by Grothendieck-Murre \cite[Theorem 9.2]{GM71}
when $A$ is two-dimensional.
Our focus in this paper is the pro-$p$ part and hence Statement \eqref{0041}.

Grothendieck originally stated \eqref{0041}
in his letter to Serre dated October 1, 1961 \cite{CS01},
where Serre's editorial note (in 2000) reads:
``I do not know whether any progress has been made on it since.''
Known results indeed seem quite scarce:
the only one the author could find (at the time this paper was submitted)
is the work of Carvajal-Rojas-Schwede-Tucker \cite{CRST18},
which proves that
$\pi_{1}(\Spec A \setminus \{\ideal{m}\})$ is finite of order prime to $p$
for the case where $A$ is a strongly $F$-regular singularity in equal characteristic $p > 0$.
(But see below about the work of Hartshorne-Speiser \cite{HS77}.)
On the other hand, wild quotient singularities constructed by Artin \cite{Art75},
Lorenzini \cite[Corollary 6.14]{Lor14} and others
give concrete examples with $\pi_{1}(\Spec A \setminus \{\ideal{m}\}) \cong \Z / p \Z$.

In this paper, we prove a weaker version of Statement \eqref{0041}
with no additional assumption on $A$:

\begin{Thm} \label{0018}
	Under the assumptions of Conjecture \ref{0017},
	the maximal pro-nilpotent quotient of $\pi_{1}(\Spec A \setminus \{\ideal{m}\})$ is
	topologically finitely generated.
\end{Thm}

In particular, the maximal abelian quotient and
the maximal pro-$p$ quotient are both topologically finitely generated.

This theorem has an application to the author's work \cite{Suz24} on arithmetic duality for $A$
when $A$ is normal and two-dimensional and has mixed characteristic.
In this work, the ``arithmetic cohomology'' $H^{q}(\Spec A \setminus \{\ideal{m}\}, \Z / p^{n} \Z(r))$
is given a canonical structure as an ind-pro-algebraic group over the residue field.
Using Theorem \ref{0018} above, we can show (done in the subsequent work \cite{SuzExpl})
that this ind-pro-algebraic group structure
actually has no connected part (that is, it is an \'etale group) when $q = 1$ and $r = 0$.
For other values of $q$ and $r$, the connected part of this arithmetic cohomology may be non-trivial.
For example, for $q = 1$ and $r = 1$, it is closely related to
Lipman's group scheme structure \cite{Lip76}
on the Picard group of a resolution of the singularity of $A$.
Thus the \'etaleness of the algebraic structure in the case $q = 1$ and $r = 0$
is a non-trivial finiteness statement.

To prove the theorem, we may assume that $A$ is normal and two-dimensional
by the same argument as the line after
\cite[Expos\'e XIII, Conjecture 3.1]{Gro05}
(using the maximal pro-nilpotent quotient of $\pi_{1}$ in place of full $\pi_{1}$).
It is enough to show that
the maximal abelian quotient of $\pi_{1}(\Spec A \setminus \{\ideal{m}\})$ is
topologically finitely generated.%
\footnote{
	Use the fact that if a finite set of elements of a pro-nilpotent group topologically generates the abelian quotient,
	then the same set topologically generates the whole group
	(\cite[Section 5.8, Lemma 5.9]{MKS04}).
}
The prime-to-$p$ part is done by Grothendieck-Murre.
Thus all we need to do is to show the finiteness of the \'etale cohomology
	\begin{equation} \label{0048}
		H^{1}(\Spec A \setminus \{\ideal{m}\}, \Z / p \Z).
	\end{equation}

The strategy is to take a resolution of singularities of $A$ and describe the $H^{1}$
by $p$-adic nearby cycles around the exceptional divisor.
These $p$-adic nearby cycles in turn are described by coherent cohomology.
The negative definiteness of intersection matrices of exceptional divisors
supplies basic bounds on the coherent cohomology.
Additionally, we need to analyze a kind of Frobenius-fixed points of some ``deformation'' cohomology.
This part is more involved in the mixed characteristic case
than in the equal characteristic case.
For this analysis, we introduce algebraic group structures on the deformation cohomology
and use their Lie algebras.

When this paper was submitted,
one of the referees pointed out that
the equal characteristic case of Theorem \ref{0018} had essentially been
obtained by Hartshorne-Speiser \cite{HS77}.
Indeed, the combination of \cite[Corollary 5.5, Theorem 5.4 and Section 2, Remark (5)]{HS77}
shows that the cohomology \eqref{0048} is finite in this case.
Their method of proof is to study general finiteness problems on Frobenius modules
and apply them to Frobenius actions on local cohomology modules.
In contrast, our proof in the the equal characteristic case relies much less on
general finiteness problems on Frobenius modules.
Our proof is also short and gives a model for the proof in the mixed characteristic case.
Therefore we believe it has its own merits and so we keep it in the original form.

This paper is organized as follows.
After some preliminaries in Section \ref{0034},
the equal characteristic $p > 0$ case is treated in Section \ref{0035}.
The rest of the paper treats the mixed characteristic case.
In Section \ref{0036},
we describe $H^{1}(\Spec A \setminus \{\ideal{m}\}, \Z / p \Z)$
by $p$-adic nearby cycles and relate it to some coherent cohomology and ``deformation cohomology''.
In Sections \ref{0010} and \ref{0037},
we give some algebraic group structures on the deformation cohomology,
calculate their Lie algebras
and show that a natural map between them is injective.
This is enough to conclude that
$H^{1}(\Spec A \setminus \{\ideal{m}\}, \Z / p \Z)$ is finite,
thus finishing the proof of Theorem \ref{0018}.

\begin{Ack}
	The author thanks the referees for the careful reading,
	especially for pointing to the work of Hartshorne-Speiser
	and to an error in arguments about intersection pairings
	in an earlier manuscript of this paper.
\end{Ack}

\begin{Notation}
	Let $A$ be a two-dimensional complete noetherian normal local ring
	with algebraically closed residue field $F$.
	Assume that $F$ has characteristic $p > 0$.
	Set $\Lambda = \Z / p \Z$.
	Let $\ideal{m}$ be the maximal ideal of $A$
	and set $X = \Spec A \setminus \{\ideal{m}\}$.
	Let $P$ be the set of height one prime ideals of $A$.
	
	Let $\mathfrak{X} \to \Spec A$ be a resolution of singularities
	such that the reduced part $Y$ of $\mathfrak{X} \times_{A} F$ is
	supported on a strict normal crossing divisor (\cite[Tag 0BIC]{Sta22}).
	Let $Y_{1}, \dots, Y_{n}$ be the irreducible components of $Y$.
	Let $I_{Y}, I_{Y_{1}}, \dots, I_{Y_{n}} \subset \Order_{\mathfrak{X}}$ be
	the ideal sheaves of $Y, Y_{1}, \dots, Y_{n}$.
	
	For an ordered set of integers $m = (m_{1}, \dots, m_{n}) \in \Z^{n}$,
	set $I_{Y}^{m} = \prod_{i} I_{Y_{i}}^{m_{i}}$.
	This notation is consistent with the $m$-th power of $I_{Y}$
	when $m_{1} = \dots = m_{n}$ and $m$ is identified with this common value of the $m_{i}$.
	That is, in this notation, we identify $\Z$ as the diagonal image in $\Z^{n}$.
	We view $\Z^{n}$ as a $\Z$-modules,
	so $m \pm m'$ for $m, m' \in \Z^{n}$ means component-wise addition/subtraction
	and $m + 1$ and $2 m$ for example means $(m_{1} + 1, \dots, m_{n} + 1)$
	and $(2 m_{1}, \dots, 2 m_{n})$.
	
	Let
		\[
			X \stackrel{j}{\into} \mathfrak{X} \stackrel{i}{\intoback} Y
		\]
	be the inclusions.
	For $q \in \Z$, let
		\begin{gather*}
					\Psi
				=
					i^{\ast} j_{\ast}
				\colon
					\Ab(X_{\et})
				\to
					\Ab(Y_{\et}),
			\\
					R^{q} \Psi
				=
					i^{\ast} R^{q} j_{\ast}
				\colon
					D(X_{\et})
				\to
					\Ab(Y_{\et}),
			\\
					R \Psi
				=
					i^{\ast} R j_{\ast}
				\colon
					D(X_{\et})
				\to
					D(Y_{\et})
		\end{gather*}
	be the nearby cycle functors
	for the categories of sheaves of abelian groups on the \'etale sites
	and their derived categories.
\end{Notation}

%%%%%%%%%%%%%%%%%%%%%%%%%%%%%%%%%%%%%%%%%%%%%%%%%%%%%%%%%%%%%%%%%%%%%%%%%%%%%

\section{Preliminaries}
\label{0034}

As discussed in Section \ref{0043},
to prove Theorem \ref{0018},
we need to show that $H^{1}(X, \Lambda)$ is finite.
By the proper base change, we have
	\[
			H^{1}(X, \Lambda)
		\cong
			H^{1}(Y, R \Psi \Lambda).
	\]
Since $H^{2}(Y, \Lambda) = 0$ by \cite[Chapter VI, Remark 1.5 (b)]{Mil80},
this induces an exact sequence
	\[
			0
		\to
			H^{1}(Y, \Lambda)
		\to
			H^{1}(X, \Lambda)
		\to
			\Gamma(Y, R^{1} \Psi \Lambda)
		\to
			0.
	\]
The group $H^{1}(Y, \Lambda)$ is finite.
Hence we are reduced to showing that
$\Gamma(Y, R^{1} \Psi \Lambda)$ is finite.

We need some preliminaries for the rest of the paper.
We will use intersection theory for exceptional curves on $\mathfrak{X}$
(\cite[Section 13]{Lip69}).
For an ordered set of integers $(m_{1}, \dots, m_{n})$,
we say that the divisor $\sum_{i} m_{i} Y_{i}$ on $\mathfrak{X}$ is \emph{nef}
(with respect to the morphism $\mathfrak{X} \to \Spec A$)
if the intersection number
$(\sum_{i} m_{i} Y_{i}) \cdot Y_{i'}$ is non-negative for all $i'$.
The negative-definiteness of intersection pairings
gives the following useful negativity properties for sheaves of the form $I_{Y}^{m'} / I_{Y}^{m}$:

\begin{Prop} \label{0049}
	Let $m = (m_{1}, \dots, m_{n})$ be
	an ordered set of non-positive integers
	such that $\sum_{i} m_{i} Y_{i}$ is nef.
	\begin{enumerate}
		\item \label{0050}
			Let $m' = (m_{1}', \dots, m_{n}')$ be an ordered set of integers
			such that $m_{i}' \le m_{i}$ for all $i$.
			Then the sheaf $I_{Y}^{m'} / I_{Y}^{m}$ admits a finite filtration
			for which every successive subquotient is supported on $Y_{i}$ for some $i$
			(which depends on the subquotient)
			giving a line bundle of negative degree on $Y_{i}$.
		\item \label{0051}
			Assume that $m_{i} \ne 0$ for any $i$.
			Then the sheaf $I_{Y}^{m} / I_{Y}^{m + 1}$ admits a finite filtration
			for which every successive subquotient is supported on $Y_{i}$ for some $i$
			giving a line bundle of negative degree on $Y_{i}$.
	\end{enumerate}
\end{Prop}

\begin{proof}
	\eqref{0050}
	There is nothing to do if $m' = m$.
	Suppose $m' \ne m$.
	For any $i'$, consider the intersection number
		\[
				\left(
					\sum_{i} (m_{i} - m_{i}') Y_{i}
				\right)
			\cdot
				Y_{i'}.
		\]
	Since $\sum_{i} (m_{i} - m_{i}') Y_{i}$ is a non-zero effective divisor,
	the negative-definiteness of the intersection matrix $(Y_{i} \cdot Y_{j})_{i j}$
	(\cite[Lemma 14.1]{Lip69}) shows that
	this number is negative for some  $i'$.
	For this $i'$, we have $m_{i'} - m_{i'}' > 0$
	since $Y_{i} \cdot Y_{j} \ge 0$ for $i \ne j$.
	Also, since $\sum_{i} m_{i} Y_{i}$ is nef,
	we have
		\begin{equation} \label{0052}
					\left(
						- \sum_{i} m_{i}' Y_{i}
					\right)
				\cdot
					Y_{i'}
			<
				0.
		\end{equation}
	Define $m'' = (m_{1}'', \dots, m_{n}'')$ by setting
	$m_{i}'' = m_{i}' + 1$ for $i = i'$
	and $m_{i}'' = m_{i}'$ otherwise.
	Then $m_{i}'' \le m_{i}$ for all $i$.
	The sheaf $I_{Y}^{m'} / I_{Y}^{m''}$ is supported on $Y_{i'}$
	giving a line bundle on $Y_{i'}$.
	Its degree is the left-hand side of \eqref{0052}
	by the definition of intersection numbers.
	Replacing $m'$ by $m''$ and doing induction,
	we get a desired filtration.
	
	\eqref{0051}
	Since $- \sum_{i} m_{i} Y_{i}$ is a non-zero effective divisor,
	the negative-definiteness of the intersection matrix shows that
		\begin{equation} \label{0053}
					\left(
						- \sum_{i} m_{i} Y_{i}
					\right)
				\cdot
					Y_{i(1)}
			<
				0
		\end{equation}
	for some $i(1)$.
	Choose $i(2) \ne i(1)$ such that $Y_{i(2)} \cap Y_{i(1)} \ne \emptyset$.
	Choose $i(3) \ne i(1), i(2)$ such that $Y_{i(3)} \cap (Y_{i(1)} \cup Y_{i(2)}) \ne \emptyset$.
	Choose $i(4) \ne i(1), i(2), i(3)$ such that $Y_{i(4)} \cap (Y_{i(1)} \cup Y_{i(2)} \cup Y_{i(3)}) \ne \emptyset$.
	Repeat until one arrives at $i(n) \ne i(1), \dots, i(n - 1)$.
	This is possible due to the connectedness of $Y$ (Zariski's main theorem).
	For $1 \le j \le n$,
	define $m(j) = (m(j)_{1}, \dots, m(j)_{n})$ by setting
		\begin{gather*}
					m(j)_{i(1)} = m_{i(1)} + 1,
				\quad \dots, \quad
					m(j)_{i(j)} = m_{i(j)} + 1,
			\\
					m(j)_{i} = m_{i}
				\text{ for }
					i \ne i(1), \dots, i(j).
		\end{gather*}
	Then $I_{Y}^{m} / I_{Y}^{m(1)}$ is supported on $Y_{i(1)}$
	giving a line bundle of negative degree by \eqref{0053}.
	Also, for $2 \le j \le n$, we have
		\begin{align*}
			&
						\left(
							- \sum_{i} m(j - 1)_{i} Y_{i}
						\right)
					\cdot
						Y_{i(j)}
			\\
			&	=
							\left(
								- \sum_{i} m_{i} Y_{i}
							\right)
						\cdot
							Y_{i(j)}
					-
						Y_{i(1)} \cdot Y_{i(j)}
					-
						\dots
					-
						Y_{i(j - 1)} \cdot Y_{i(j)}.
		\end{align*}
	The first term of the right-hand side is non-positive
	by the nefness of $\sum_{i} m_{i} Y_{i}$.
	All other terms are non-positive and at least one of them is negative
	since $Y_{i(j)} \cap (Y_{i(1)} \cup \dots \cup Y_{i(j - 1)}) \ne \emptyset$.
	Thus $I_{Y}^{m(j - 1)} / I_{Y}^{m(j)}$ is supported on $Y_{i(j)}$
	giving a line bundle of negative degree.
	Since $m(n) = m + 1$, this gives a desired filtration.
\end{proof}

%%%%%%%%%%%%%%%%%%%%%%%%%%%%%%%%%%%%%%%%%%%%%%%%%%%%%%%%%%%%%%%%%%%%%%%%%%%%%

\section{Equal characteristic case}
\label{0035}

In this section, we treat the equal characteristic case.
Assume that the characteristic of the fraction field of $A$ is $p$ ($> 0$).
Let $\Order_{X} = \Ga \in \Ab(X_{\et})$ be
the structure sheaf of $X$ in the \'etale topology.
Let $\Frob \colon \Order_{X} \to \Order_{X}$ be the $p$-th power map.
Since $j$ is an affine morphism,
we have $R^{q} j_{\ast} \Order_{X} = 0$ for $q \ge 1$.
Hence the Artin-Schreier sequence
	\[
		0 \to \Lambda \to \Order_{X} \stackrel{\Frob - 1}{\to} \Order_{X} \to 0
	\]
over $X_{\et}$ induces an exact sequence
	\begin{equation} \label{0020}
			0
		\to
			\Lambda
		\to
			\Psi \Order_{X}
		\stackrel{\Frob - 1}{\to}
			\Psi \Order_{X}
		\to
			R^{1} \Psi \Lambda
		\to
			0
	\end{equation}
over $Y_{\et}$.
On the other hand, for any closed point $x \in Y$,
we have an Artin-Schreier exact sequence
	\[
			0
		\to
			\Lambda
		\to
			A_{x}
		\stackrel{\Frob - 1}{\to}
			A_{x}
		\to
			0,
	\]
where $A_{x}$ is the (strict) henselian local ring of $\mathfrak{X}$ at $x$.
Hence we have an exact sequence
	\begin{equation} \label{0021}
			0
		\to
			\Lambda
		\to
			i^{\ast} \Order_{\mathfrak{X}}
		\stackrel{\Frob - 1}{\to}
			i^{\ast} \Order_{\mathfrak{X}}
		\to
			0
	\end{equation}
over $Y_{\et}$ (where this $i^{\ast}$ is the pullback for abelian sheaves, not coherent sheaves).
Combining \eqref{0020} and \eqref{0021},
we obtain an exact sequence
	\[
			0
		\to
			\Psi \Order_{X} / i^{\ast} \Order_{\mathfrak{X}}
		\stackrel{\Frob - 1}{\to}
			\Psi \Order_{X} / i^{\ast} \Order_{\mathfrak{X}}
		\to
			R^{1} \Psi \Lambda
		\to
			0
	\]
over $Y_{\et}$.
The sheaf $j_{\ast} \Order_{X}$ is the union of the subsheaves
$I_{Y}^{- m}$ over integers $m \ge 1$.
Hence
	\[
			\Psi \Order_{X} / i^{\ast} \Order_{\mathfrak{X}}
		\cong
			\bigcup_{n}
				I_{Y}^{- m} / \Order_{\mathfrak{X}},
	\]
where we omit $i^{\ast}$ from the right-hand side
as $I_{Y}^{- m} / \Order_{\mathfrak{X}}$ is supported on $Y$.

By \cite[Lemma (4.5)]{Sai87}
(namely by the negative definiteness of the intersection pairing
and the numerical criterion of ampleness \cite[Theorem (12.1) (iii)]{Lip69}),
there exists an ordered set of positive integers $c = (c_{1}, \dots, c_{n})$
such that $J = I_{Y}^{c} \subset \Order_{\mathfrak{X}}$ is ample.
For any integer $m \ge 1$, define $U_{J}^{m} R^{1} \Psi \Lambda$ to be
the image of $J^{- p m} / \Order_{\mathfrak{X}}$ in $R^{1} \Psi \Lambda$.
We have an exact sequence
	\begin{equation} \label{0024}
			0
		\to
			J^{- m} / \Order_{\mathfrak{X}}
		\stackrel{\Frob - 1}{\to}
			J^{- p m} / \Order_{\mathfrak{X}}
		\to
			U_{J}^{m} R^{1} \Psi \Lambda
		\to
			0.
	\end{equation}
Let
	\[
			\gr_{J}^{m} R^{1} \Psi \Lambda
		=
				U_{J}^{m} R^{1} \Psi \Lambda
			/
				U_{J}^{m - 1} R^{1} \Psi \Lambda.
	\]

\begin{Prop} \label{0044}
	We have $\Gamma(Y, \gr_{J}^{m} R^{1} \Psi \Lambda) = 0$ for $m \gg 0$.
\end{Prop}

\begin{proof}
	Assume $m \ge 2$.
	Then $J^{- m} \subset J^{- p m + p}$.
	Hence \eqref{0024} induces an exact sequence
		\[
				0
			\to
				J^{- m} / J^{- m + 1}
			\stackrel{\Frob}{\to}
				J^{- p m} / J^{- p m + p}
			\to
				\gr_{J}^{m} R^{1} \Psi \Lambda
			\to
				0.
		\]
	Let $Y_{J} \subset \mathfrak{X}$ be the closed subscheme defined by the ideal $J$
	and $Y_{J^{p}} \subset \mathfrak{X}$ similarly by $J^{p}$.
	The $p$-th power map $\Order_{\mathfrak{X}} / J \into \Order_{\mathfrak{X}} / J^{p}$
	defines a morphism $\Frob \colon Y_{J^{p}} \to Y_{J}$.
	We have a cartesian diagram
		\begin{equation} \label{0045}
			\begin{CD}
					Y_{J^{p}}
				@> \Frob >>
					Y_{J}
				\\ @V \text{incl} VV @VV \text{incl} V \\
					\mathfrak{X}
				@>> \Frob >
					\mathfrak{X},
			\end{CD}
		\end{equation}
	where the lower horizontal morphism is the absolute Frobenius.
	With $J^{- m} / J^{- m + 1}$ viewed as a line bundle on $Y_{J}$
	and $J^{- p m} / J^{- p m + p}$ as a line bundle on $Y_{J^{p}}$,
	the morphism
		$
				J^{- m} / J^{- m + 1}
			\stackrel{\Frob}{\to}
				J^{- p m} / J^{- p m + p}
		$
	above can be viewed as a morphism
	$J^{- m} / J^{- m + 1} \to \Frob_{\ast}(J^{- p m} / J^{- p m + p})$
	of $\Order_{Y_{J}}$-modules.
	This last morphism can be obtained from the inclusion
	$\Order_{Y_{J}} \into \Frob_{\ast} \Order_{Y_{J^{p}}}$
	by the tensor product with $J^{- m} / J^{- m + 1}$ over $\Order_{Y_{J}}$.
	Hence
		\begin{equation} \label{0046}
				\gr_{J}^{m} R^{1} \Psi \Lambda
			\cong
					((\Frob_{\ast} \Order_{Y_{J^{p}}}) / \Order_{Y_{J}})
				\tensor_{\Order_{Y_{J}}}
					(J^{- m} / J^{- m + 1}).
		\end{equation}
	The regularity of $\mathfrak{X}$ implies that
	the sheaf $(\Frob_{\ast} \Order_{\mathfrak{X}}) / \Order_{\mathfrak{X}}$
	is locally free of finite rank.
	Hence by \eqref{0045}, we know that
	$((\Frob_{\ast} \Order_{Y_{J^{p}}}) / \Order_{Y_{J}})$
	is locally free of finite rank over $\Order_{Y_{J}}$.
	Since $Y_{J}$ is a local complete intersection
	and hence Cohen-Macaulay,
	it has no embedded point (\cite[Tag 0BXG]{Sta22}).
	As $J^{- m} / J^{- m + 1}$ is the ($- m$)-th tensor power
	of the ample line bundle $J / J^{2}$ on $Y_{J}$,
	this implies that
	\eqref{0046} has no non-zero global section for $m \gg 0$
	by \cite[Tag 0FD7]{Sta22}.
\end{proof}

\begin{Prop} \label{0047}
	The group $\Gamma(Y, U_{J}^{m} R^{1} \Psi \Lambda)$ is finite for all $m \ge 1$
\end{Prop}

\begin{proof}
	By Proposition \ref{0049} \eqref{0050},
	we have $\Gamma(Y, J^{- p m} / \Order_{\mathfrak{X}}) = 0$.
	Hence the sequence \eqref{0024} induces an exact sequence
		\begin{equation} \label{0039}
				0
			\to
				\Gamma(Y, U_{J}^{m} R^{1} \Psi \Lambda)
			\to
				H^{1}(Y, J^{- m} / \Order_{\mathfrak{X}})
			\stackrel{\Frob - 1}{\to}
				H^{1}(Y, J^{- p m} / \Order_{\mathfrak{X}}).
		\end{equation}
	The morphisms
		\begin{equation} \label{0025}
				\Frob \; \text{and} \; 1
			\colon
				H^{1}(Y, J^{- m} / \Order_{\mathfrak{X}})
			\to
				H^{1}(Y, J^{- p m} / \Order_{\mathfrak{X}}).
		\end{equation}
	are Frobenius-linear and linear maps, respectively,
	of finite-dimensional $F$-vector spaces.
	Consider the exact sequence
		\[
				0
			\to
				J^{- m} / \Order_{\mathfrak{X}}
			\to
				J^{- p m} / \Order_{\mathfrak{X}}
			\to
				J^{- p m} / J^{- m}
			\to
				0.
		\]
	The ampleness of $J$ shows that
	the divisor $- m \sum_{i} c_{i} Y_{i}$ has positive intersection with every $Y_{i'}$
	and, in particular, is nef.
	Hence Proposition \ref{0049} \eqref{0050} shows that
	$\Gamma(Y, J^{- p m} / J^{- m}) = 0$.%
	\footnote{
		Alternatively, we may assume that $m \gg 1$
		and can see that $\Gamma(Y, J^{- p m} / J^{- m}) = 0$ for $m \gg 1$
		by the same argument as the last part of the proof of Proposition \ref{0044}.
	}
	Hence the morphism $1$ in \eqref{0025} is injective.
	Now the result follows from Lemma \ref{0058} below.
\end{proof}

\begin{Lem} \label{0058}
	Let $V$ and $W$ be finite-dimensional $F$-vector spaces.
	Let $f \colon V \to W$ and $g \colon V \to W$ be Frobenius-linear
	and linear maps, respectively.
	Assume that $g$ is injective.
	Then the kernel of $f - g$ is a finite group.
\end{Lem}

\begin{proof}
	Take a linear splitting $p \colon W \onto \Im(g)$ of the natural inclusion $\Im(g) \into W$.
	Then $\Ker(f - g) \subset \Ker(p \compose (f - g))$.
	Hence we may replace $f$ and $g$ by $p \compose f$ and $p \compose g$, respectively.
	Then $g$ is bijective.
	In this case, the lemma is \cite[Tag 0A3L]{Sta22}.
\end{proof}

\begin{Prop}
	The group $\Gamma(Y, R^{1} \Psi \Lambda)$ is finite.
\end{Prop}

\begin{proof}
	By Proposition \ref{0044},
	for $m$ large enough, we have
	$\Gamma(Y, R^{1} \Psi \Lambda / U_{J}^{m} R^{1} \Psi \Lambda) = 0$.
	Since $\Gamma(Y, U_{J}^{m} R^{1} \Psi \Lambda)$ is finite by
	Proposition \ref{0047},
	it follows that $\Gamma(Y, R^{1} \Psi \Lambda)$ is finite.
\end{proof}

This proves Theorem \ref{0018} in the equal characteristic case.

%%%%%%%%%%%%%%%%%%%%%%%%%%%%%%%%%%%%%%%%%%%%%%%%%%%%%%%%%%%%%%%%%%%%%%%%%%%%%

\section{%
	\texorpdfstring{Mixed characteristic case: $p$-adic nearby cycles}
	{Mixed characteristic case: p-adic nearby cycles}
}
\label{0036}

For the rest of the paper, we treat the mixed characteristic case.
Assume that the characteristic of the fraction field of $A$ is zero.
We may assume that $A$ contains a fixed primitive $p$-th root of unity $\zeta_{p}$.
Let $S = \{\ideal{p}_{1}, \dots, \ideal{p}_{l}\} \subset P$ be
the set of all height one prime ideals containing $p$ and set $U = X \setminus S = \Spec A[1 / p]$.
For $\ideal{p} \in P$, let $A_{\ideal{p}}$ be the henselian local ring of $A$ at $\ideal{p}$.
Let $K_{\ideal{p}}$ the fraction field of $A_{\ideal{p}}$
and $\kappa(\ideal{p})$ its residue field.
If $\ideal{p} \in S$,
then let $e_{\ideal{p}}$ be the absolute ramification index of $A_{\ideal{p}}$
and set $f_{\ideal{p}} = p e_{\ideal{p}} / (p - 1)$.
Note that the facts $\zeta_{p} \in A$ and $[\Q_{p}(\zeta_{p}) : \Q_{p}] = p - 1$ imply that
$e_{\ideal{p}} \in (p - 1) \Z$ and hence $f_{\ideal{p}} \in p \Z$.
Define $\mathfrak{T}(1)$ to be the complex of \'etale sheaves $0 \to \Gm \to \Gm \to 0$ on $X$
with non-zero terms in degrees $0$ and $1$
given by the $p$-th power map.
View it as an object of the derived category $D(X_{\et})$.
The morphism $\Lambda \to \Gm$ sending $1$ to $\zeta_{p}$ defines a morphism
	\begin{equation} \label{0000}
		\Lambda \to \mathfrak{T}(1)
	\end{equation}
in $D(X_{\et})$.
It is an isomorphism over $U$.

By suitably replacing the resolution $\mathfrak{X} \to \Spec A$,
we may assume that $Y \cup Z \subset \mathfrak{X}$ is
supported on a strict normal crossing divisor,
where $Z$ is the (reduced) closure of $S$ in $\mathfrak{X}$ (\cite[Tag 0BIC]{Sta22}).
For a closed point $x \in Y$,
let $A_{x}$ and $B_{x}$ be the henselian local rings of $\mathfrak{X}$ and $Y$, respectively, at $x$.
The ring $A_{x}$ is regular and hence a UFD
by the Auslander-Buchsbaum theorem.
Let $R_{x}$ be the affine ring of $\Spec A_{x} \times_{\mathfrak{X}} X$
and let $R'_{x} = A_{x}[1 / p]$
(which are regular UFD's).
Let $I_{Z} \subset \Order_{\mathfrak{X}}$ be the ideal sheaf of $Z$.
For each $j$, let $Z_{j}$ be the closure of $\ideal{p}_{j} \in S$ in $\mathfrak{X}$.
Let $I_{Z_{j}} \subset \Order_{\mathfrak{X}}$ be the ideal sheaf of $Z_{j}$.
For an ordered set of integers $m' = (m'_{1}, \dots, m'_{l})$,
let $I_{Z}^{m'} = \prod_{j} I_{Y_{j}}^{m_{j}'}$.
Let $A_{Y_{i}}$ be the henselian local ring of $\mathfrak{X}$
at the generic point of $Y_{i}$.
Let $K_{Y_{i}}$ be its fraction field.
Let $e_{Y_{i}}$ be its absolute ramification index
and set $f_{Y_{i}} = p e_{Y_{i}} / (p - 1)$.
Set $f_{Y} = (f_{Y_{1}}, \dots, f_{Y_{n}})$.
Set $e_{Z_{j}} = e_{\ideal{p}_{j}}$ and $f_{Z_{j}} = f_{\ideal{p}_{j}}$.
Set $f_{Z} = (f_{Z_{1}}, \dots, f_{Z_{l}})$.
Again, we have $e_{Y_{i}}, e_{Z_{j}} \in (p - 1) \Z$
and $f_{Y_{i}}, f_{Z_{j}} \in p \Z$.

The stalk of $R^{1} \Psi \Gm$ at a closed point $x \in Y$
is $H^{1}(R_{x}, \Gm)$, which is zero since $R_{x}$ is a UFD.
It follows that $R^{1} \Psi \Gm = 0$.
Hence $\Psi \Gm / p \Psi \Gm \isomto R^{1} \Psi \mathfrak{T}(1)$.
For ordered sets of non-negative integers
$m = (m_{1}, \dots, m_{n})$ and $m' = (m'_{1}, \dots, m'_{l})$ not all zero,
we have a subsheaf $1 + I_{Y}^{m} I_{Z}^{m'}$ of $\Gm$ on $\mathfrak{X}_{\et}$.
Define $U^{(m, m')} R^{1} \Psi \mathfrak{T}(1)$ to be
the image of $i^{\ast}(1 + I_{Y}^{m} I_{Z}^{m'})$ ($\subset i^{\ast} \Gm \subset \Psi \Gm$)
in $R^{1} \Psi \mathfrak{T}(1)$.

\begin{Prop}
	The morphism \eqref{0000} induces an isomorphism
		\[
				R^{1} \Psi \Lambda
			\isomto
				U^{(0, f_{Z})} R^{1} \Psi \mathfrak{T}(1)
			\quad
			(\subset
				R^{1} \Psi \mathfrak{T}(1)
			).
		\]
\end{Prop}

\begin{proof}
	On $Y \setminus Z$, both sides are isomorphic to the whole $R^{1} \Psi \mathfrak{T}(1)$.
	Let $x \in Y \cap Z$.
	Let $\ideal{p} = \ideal{p}_{j} \in S$ be its unique generalization in $S$.
	The inverse image of $Z_{j}$ under $\Spec A_{x} \to \mathfrak{X}$
	corresponds to a prime ideal $\ideal{p}_{x} \in \Spec A_{x}$.
	The stalks of $R^{1} \Psi \Lambda$ and $R^{1} \Psi \mathfrak{T}(1)$ at $x$ are given by
	$H^{1}(R_{x}, \Lambda)$ and
	$H^{1}(R_{x}, \mathfrak{T}(1)) \cong R_{x}^{\times} / R_{x}^{\times p}$,
	respectively.
	Localization gives a commutative diagram with exact rows
		\[
			\begin{CD}
					0
				@>>>
					H^{1}(R_{x}, \Lambda)
				@>>>
					H^{1}(R_{x}', \Lambda)
				@>>>
						H^{1}(K_{\ideal{p}}, \Lambda)
					/
						H^{1}(\kappa(\ideal{p}), \Lambda)
				\\
				@. @VVV @| @.
				\\
					0
				@>>>
					H^{1}(R_{x}, \mathfrak{T}(1))
				@>>>
					H^{1}(R_{x}', \mathfrak{T}(1)).
				@.
			\end{CD}
		\]
	Under the isomorphism
	$H^{1}(K_{\ideal{p}}, \Lambda) \cong K_{\ideal{p}}^{\times} / K_{\ideal{p}}^{\times p}$,
	the subgroup $H^{1}(\kappa(\ideal{p}), \Lambda)$ is identified with
	the image of $1 + \ideal{p}^{f_{\ideal{p}}} A_{\ideal{p}}$
	by Lemma \ref{0059} below.
	Hence this diagram implies that $H^{1}(R_{x}, \Lambda)$ is identified with
	the subgroup of $H^{1}(R_{x}', \mathfrak{T}(1))$ that maps into
	the image of $1 + \ideal{p}^{f_{\ideal{p}}} A_{\ideal{p}}$
	in $K_{\ideal{p}}^{\times} / K_{\ideal{p}}^{\times p}$.
	By Lemma \ref{0031} below, this subgroup is the image of
	$1 + \ideal{p}_{x}^{f_{\ideal{p}}} A_{x}$ in $R_{x}^{\times} / R_{x}^{\times p}$.
\end{proof}

\begin{Lem} \label{0059}
	Under the isomorphisms
	$H^{1}(K_{\ideal{p}}, \Lambda) \cong K_{\ideal{p}}^{\times} / K_{\ideal{p}}^{\times p}$
	and $H^{1}(\kappa(\ideal{p}), \Lambda) \cong \kappa(\ideal{p}) / (\Frob - 1) \kappa(\ideal{p})$,
	the map $H^{1}(\kappa(\ideal{p}), \Lambda) \into H^{1}(K_{\ideal{p}}, \Lambda)$
	is given by sending $a \in \kappa(\ideal{p})$
	to $1 + (\zeta_{p} - 1)^{p} \Tilde{a}$,
	where $\Tilde{a}$ is any lift of $a$ to $A_{\ideal{p}}$.
\end{Lem}

\begin{proof}
	For an indeterminate $z$,
	the polynomial
		\[
			\frac{
				(1 + (\zeta_{p} - 1) z)^{p} - 1
			}{
				(\zeta_{p} - 1)^{p}
			}
		\]
	has coefficients in $\Z[\zeta_{p}]$
	whose image in $\F_{p}$ (or reduction) is $z^{p} - z$.
	Hence the Artin-Schreier equation $z^{p} - z = a$ over $\kappa(\ideal{p})$
	lifts to the Kummer equation
	$(1 + (\zeta_{p} - 1) z)^{p} = 1 + (\zeta_{p} - 1)^{p} \Tilde{a}$ over $K_{\ideal{p}}$.
	The Galois action $z \mapsto z + 1$ over $\kappa(\ideal{p})$ corresponds to
	the Galois action $1 + (\zeta_{p} - 1) z \mapsto \zeta_{p} (1 + (\zeta_{p} - 1) z)$ over $K_{\ideal{p}}$.
\end{proof}

\begin{Lem} \label{0031}
	For $m \ge 1$, let $U_{R_{x}}^{(m)}$ be the image of $1 + \ideal{p}_{x}^{m} A_{x}$
	in $U_{R_{x}}^{(0)} = R_{x}^{\times} / R_{x}^{\times p}$
	and let $U_{K_{\ideal{p}}}^{(m)}$ be the image of $1 + \ideal{p}^{m} A_{\ideal{p}}$
	in $U_{K_{\ideal{p}}}^{(0)} = K_{\ideal{p}}^{\times} / K_{\ideal{p}}^{\times p}$.
	Then the natural map
		\[
				U_{R_{x}}^{(m)} / U_{R_{x}}^{(m + 1)}
			\to
				U_{K_{\ideal{p}}}^{(m)} / U_{K_{\ideal{p}}}^{(m + 1)}
		\]
	is injective for $0 \le m \le f_{\ideal{p}} - 1$.
\end{Lem}

\begin{proof}
	These graded pieces can be explicitly calculated;
	see \cite[Section 4]{Sai86} for example.
\end{proof}

\begin{Prop} \label{0023}
	The inclusion
		\[
				\Gamma \bigl(
					Y, U^{(1, f_{Z})} R^{1} \Psi \mathfrak{T}(1)
				\bigr)
			\into
				\Gamma \bigl(
					Y, U^{(0, f_{Z})} R^{1} \Psi \mathfrak{T}(1)
				\bigr)
		\]
	has finite cokernel.
\end{Prop}

\begin{proof}
	It is enough to show that the group
		\[
			\Gamma \left(
				Y,\,
				\frac{
					U^{(0, f_{Z})} R^{1} \Psi \mathfrak{T}(1)
				}{
					U^{(1, f_{Z})} R^{1} \Psi \mathfrak{T}(1)
				}
			\right)
		\]
	is finite.
	Let $U^{(0, 0)} R^{1} \Psi \mathfrak{T}(1)$ be
	the image of $i^{\ast} \Gm$ in $R^{1} \Psi \mathfrak{T}(1)$.
	Consider the natural morphism
		\begin{equation} \label{0032}
				\frac{
					U^{(0, f_{Z})} R^{1} \Psi \mathfrak{T}(1)
				}{
					U^{(1, f_{Z})} R^{1} \Psi \mathfrak{T}(1)
				}
			\to
				\frac{
					U^{(0, 0)} R^{1} \Psi \mathfrak{T}(1)
				}{
					U^{(1, 0)} R^{1} \Psi \mathfrak{T}(1)
				}
		\end{equation}
	of sheaves on $Y_{\et}$.
	For any closed point $x \in Y$,
	the induced morphism on the stalks at $x$ is given by
		\[
					(1 + \ideal{p}_{x}^{f_{\ideal{p}}} B_{x})
				/
					(1 + \ideal{p}_{x}^{f_{\ideal{p}} / p} B_{x})^{p}
			\into
				B_{x}^{\times} / B_{x}^{\times p}.
		\]
	Hence \eqref{0032} is injective
	and the right-hand side of \eqref{0032} is isomorphic to $\Gm / \Gm^{p}$.
	Now the finiteness of the group
		\[
				\Gamma(Y, \Gm / \Gm^{p})
			\cong
				H^{1}(Y, \Gm)[p]
			\cong
				\bigoplus_{i} \Pic(Y_{i})[p]
		\]
	implies the result.
\end{proof}

Thus we need to prove that 
	$
		\Gamma \bigl(
			Y, U^{(1, f_{Z})} R^{1} \Psi \mathfrak{T}(1)
		\bigr)
	$
is finite.

\begin{Prop}
	The $p$-th power map and the natural surjection give an exact sequence
		\begin{equation} \label{0005}
				0
			\to
				i^{\ast}
				\frac{
					1 + I_{Y} I_{Z}^{f_{Z} / p}
				}{
					1 + I_{Y}^{f_{Y} / p} I_{Z}^{f_{Z} / p}
				}
			\stackrel{p}{\to}
				i^{\ast}
				\frac{
					1 + I_{Y} I_{Z}^{f_{Z}}
				}{
					1 + I_{Y}^{f_{Y}} I_{Z}^{f_{Z}}
				}
			\to
				U^{(1, f_{Z})} R^{1} \Psi \mathfrak{T}(1)
			\to
				0
		\end{equation}
	over $Y_{\et}$,
	where $f_{Y} / p$ means $(f_{Y_{1}} / p, \dots, f_{Y_{n}} / p)$
	and $f_{Z} / p$ similarly.
\end{Prop}

\begin{proof}
	We will only prove that $i^{\ast}(1 + I_{Y}^{f_{Y}} I_{Z}^{f_{Z}})$ maps to zero
	in $U^{(1, f_{Z})} R^{1} \Psi \mathfrak{T}(1)$.
	The rest of the claimed exactness can be proven by a similar method.
	
	It is enough to look at the stalk at an arbitrary closed point $x \in Y$.
	The stalk at $x$ of $R^{1} \Psi \mathfrak{T}(1)$ is
	$H^{1}(R_{x}, \mathfrak{T}(1)) \cong R_{x}^{\times} / R_{x}^{\times p}$.
	The sheaves $I_{Y}$ and $I_{Z}$ correspond to ideals of $A_{x}$.
	Hence we need to show that an element of $1 + I_{Y}^{f_{Y}} I_{Z}^{f_{Z}}$
	is a $p$-th power in $R_{x}^{\times}$.
	We will actually prove that it is a $p$-th power of an element of
	$1 + I_{Y}^{f_{Y} / p} I_{Z}^{f_{Z} / p}$.
	By considering a prime factorization of $\zeta_{p} - 1$ in the UFD $A_{x}$,
	we know that the ideal $I_{Y}^{f_{Y} / p} I_{Z}^{f_{Z} / p}$ of $A_{x}$ is
	generated by $\zeta_{p} - 1$.
	Let $a \in A_{x}$ be arbitrary and $z$ an indeterminate.
	We need to show that the equation
	$(1 + (\zeta_{p} - 1) z)^{p} = 1 + (\zeta_{p} - 1)^{p} a$
	has a solution $z \in A_{x}$.
	By the proof of Lemma \ref{0059} and the Henselian property of $A_{x}$,
	this is equivalent to solve $z^{p} - z = \bar{a}$ in $F$.
	As $F$ is algebraically closed, this equation is indeed solvable.
\end{proof}

\begin{Prop} \label{0054}
	Let $c$ be a positive integer.
	For each $i$ and j, set $f_{Y_{i}}^{\ast} = c f_{Y_{i}} / p$,
	$f_{Y}^{\ast} = (f_{Y_{1}}^{\ast}, \dots, f_{Y_{n}}^{\ast})$,
	$f_{Z_{j}}^{\ast} = c f_{Z_{j}} / p$ and
	$f_{Z}^{\ast} = (f_{Z_{1}}^{\ast}, \dots, f_{Z_{l}}^{\ast})$.
	Let $m = (m_{1}, \dots, m_{n})$ be
	an ordered set of non-positive integers
	such that $\sum_{i} m_{i} Y_{i}$ is nef.
	\begin{enumerate}
		\item \label{0055}
			Let $m' = (m_{1}', \dots, m_{n}')$ be an ordered set of integers
			such that $m_{i}' \le m_{i}$ for all $i$.
			Then the sheaf
			$I_{Y}^{f_{Y}^{\ast} + m'} I_{Z}^{f_{Z}^{\ast}} / I_{Y}^{f_{Y}^{\ast} + m} I_{Z}^{f_{Z}^{\ast}}$
			admits a finite filtration
			for which every successive subquotient is supported on $Y_{i}$ for some $i$
			giving a line bundle of negative degree on $Y_{i}$.
		\item \label{0056}
			Assume that $m_{i} \ne 0$ for any $i$.
			Then the sheaf
			$I_{Y}^{f_{Y}^{\ast} + m} I_{Z}^{f_{Z}^{\ast}} / I_{Y}^{f_{Y}^{\ast} + m + 1} I_{Z}^{f_{Z}^{\ast}}$
			admits a finite filtration
			for which every successive subquotient is supported on $Y_{i}$ for some $i$
			giving a line bundle of negative degree on $Y_{i}$.
		\item \label{0057}
			The divisor $- \sum_{i} f_{Y_{i}}^{\ast} Y_{i}$ is nef.
	\end{enumerate}
\end{Prop}

\begin{proof}
	We have $I_{Y}^{e_{Y}} I_{Z}^{e_{Z}} = p \Order_{\mathfrak{X}} \cong \Order_{\mathfrak{X}}$.
	Hence if $m_{i}' = m_{i} - 1$ for exactly one $i = i'$
	and $m_{i}' = m_{i}$ for $i \ne i'$, then the sheaf
	$I_{Y}^{f_{Y}^{\ast} + m'} I_{Z}^{f_{Z}^{\ast}} / I_{Y}^{f_{Y}^{\ast} + m} I_{Z}^{f_{Z}^{\ast}}$
	on $Y_{i'}$ has negative degree if and only if the sheaf
	$I_{Y}^{m'} / I_{Y}^{m}$ on $Y_{i'}$ has negative degree.
	Hence \eqref{0055} and \eqref{0056} follow from
	Proposition \ref{0049}.
	For \eqref{0057}, we have
		\[
				\left(
					- \sum_{i} f_{Y_{i}}^{\ast} Y_{i}
				\right) \cdot Y_{i'}
			=
				\left(
					\sum_{j} f_{Z_{j}}^{\ast} Z_{j}
				\right) \cdot Y_{i'}
			\ge
				0
		\]
	since $Z_{j} \cdot Y_{i'} \ge 0$.
\end{proof}

The following gives a mixed characteristic analogue of the sequence \eqref{0039}:

\begin{Prop} \label{0016}
	We have
		\[
				\Gamma \left(
					\mathfrak{X},\,
					\frac{
						1 + I_{Y} I_{Z}^{f_{Z}}
					}{
						1 + I_{Y}^{f_{Y}} I_{Z}^{f_{Z}}
					}
				\right)
			=
				0.
		\]
	In particular, the sequence \eqref{0005} induces an exact sequence
		\[
				0
			\to
				\Gamma \bigl(
					Y,
					U^{(1, f_{Z})} R^{1} \Psi \mathfrak{T}(1)
				\bigr)
			\to
				H^{1} \left(
					\mathfrak{X},\,
					\frac{
						1 + I_{Y} I_{Z}^{f_{Z} / p}
					}{
						1 + I_{Y}^{f_{Y} / p} I_{Z}^{f_{Z} / p}
					}
				\right)
			\stackrel{p}{\to}
				H^{1} \left(
					\mathfrak{X},\,
					\frac{
						1 + I_{Y} I_{Z}^{f_{Z}}
					}{
						1 + I_{Y}^{f_{Y}} I_{Z}^{f_{Z}}
					}
				\right)
		\]
\end{Prop}

\begin{proof}
	By Proposition \ref{0054} \eqref{0055}, the sheaf
	$I_{Y} I_{Z}^{f_{Z}} / I_{Y}^{f_{Y}} I_{Z}^{f_{Z}}$
	admits a finite filtration whose successive subquotients are
	line bundles of negative degree on some of $Y_{1}, \dots, Y_{n}$.
	Hence $i^{\ast} (1 + I_{Y} I_{Z}^{f_{Z}}) / (1 + I_{Y}^{f_{Y}} I_{Z}^{f_{Z}})$
	admits a finite filtration whose successive subquotients are
	line bundles of negative degree on some of $Y_{1}, \dots, Y_{n}$.
	Hence its global section module is zero.
\end{proof}

We want to prove that the map $p$ in the proposition has finite kernel.
The difference between this sequence and \eqref{0039}
is that the cohomology groups are no longer $F$-vector spaces,
not even killed by $p$.
Note, however, that $(1 + a)^{p} = 1 + p a + \dots + a^{p}$,
so the map $p$ is not very different from $\Frob - 1$.

The strategy is to give some algebraic group structures (instead of vector space structures) on the $H^{1}$
and show that the map $p$ on their Lie algebras is injective
and hence $p$ itself has finite \'etale kernel.
This map on the Lie algebras should be something like
	\begin{equation} \label{0040}
			H^{1} \left(
				\mathfrak{X},\,
				\frac{
					I_{Y} I_{Z}^{f_{Z} / p}
				}{
					I_{Y}^{f_{Y} / p} I_{Z}^{f_{Z} / p}
				}
			\right)
		\stackrel{p}{\to}
			H^{1} \left(
				\mathfrak{X},\,
				\frac{
					I_{Y} I_{Z}^{f_{Z}}
				}{
					I_{Y}^{f_{Y}} I_{Z}^{f_{Z}}
				}
			\right),
	\end{equation}
where this $p$ is induced by multiplication by $p$ on the coefficient sheaves.
While we can see using Proposition \ref{0054} that the map \eqref{0040} is indeed injective,
it cannot be a map between Lie algebras of some algebraic groups over $F$
since the groups in \eqref{0040} are not killed by $p$.
It turns out that some twists are necessary
and we need to kill some ``junk'' infinitesimal group schemes
that unnecessarily fatten up the Lie algebras.
We will carry out this strategy in the subsequent sections.

%%%%%%%%%%%%%%%%%%%%%%%%%%%%%%%%%%%%%%%%%%%%%%%%%%%%%%%%%%%%%%%%%%%%%%%%%%%%%

\section{Lie algebras of deformation cohomology I}
\label{0010}

We will use the methods of \cite{Lip76}
to treat the type of cohomology appearing in Proposition \ref{0016}.
We first treat an algebraic group structure on the latter group
$H^{1} \bigl( \mathfrak{X}, (1 + I_{Y} I_{Z}^{f_{Z}}) / (1 + I_{Y}^{f_{Y}} I_{Z}^{f_{Z}}) \bigr)$.
The former group is treated in the next section.
Some basic references for commutative algebraic groups as fppf sheaves
are \cite{DG70DemaGab}, \cite{DG70SGA} and \cite{Oor66}.

We need some notation.
For a commutative associative nilpotent ring $I$ without unity
(nilpotent means that any element $a$ satisfies $a^{n} = 0$ for some $n = n(a)$),
define a group $1 + I$ to be the set $I$ with new group structure given by
$a \cdot b = a + b + a b$.
An element of $I$ viewed as an element of this $1 + I$ is denoted by $1 + a$.
Note that if $I^{2} = 0$ (meaning $a b = 0$ for all $a, b \in I$),
then $1 + I$ is isomorphic to the additive group of $I$.

For $N \ge 0$, let $W_{N}(F)$ be the ring of $p$-typical Witt vectors of length $N$.
For the associative $W_{N}(F)$-algebras without unity below,
we assume that $1 \in W_{N}(F)$ acts by multiplication by the identity map.
The ring $A$ has a canonical structure as a $W(F) = \invlim_{N} W_{N}(F)$-algebra
(\cite[Chapter V, Section 4, Theorem 2.1]{DG70DemaGab}).
Hence $\mathfrak{X}$ is naturally a $W(F)$-scheme.

In this paper, the fppf site $\Spec F_{\fppf}$ of $F$ is
(the opposite of) the category of $F$-algebras
endowed with the (pre)topology where a covering of an $F$-algebra $R$ is
a finite family $\{R \to R_{i}\}_{i}$ of flat morphisms of finite presentation
such that $R \to \prod R_{i}$ is faithfully flat.
This is called the big affine fppf site of $\Spec F$ in \cite[Tag 021S]{Sta22}.
A sheaf on it is simply called a ``sheaf''
in \cite[Chapter III, Section 1, No.\ 1.2]{DG70DemaGab}.
The affineness is just for simplicity:
its topos is equivalent to the topos of the ``usual'' fppf site of $F$,
which is the category of $F$-schemes
endowed with the (pre)topology where a covering of an $F$-scheme $S$ is
a family $\{S_{\lambda} \to S\}_{\lambda}$ of flat morphisms locally of finite presentation
such that $\bigsqcup S_{\lambda} \to S$ is surjective (\cite[Tag 021V]{Sta22}).
As an intermediate generalization,
one may also restrict the schemes to be separated,
which again results in an equivalent topos.

For integers $q, N \ge 0$,
a sheaf $I$ of commutative associative nilpotent $W_{N}(F)$-algebras without unity over $Y_{\et}$
and a sheaf $J$ of commutative associative $W_{N}(F)$-algebras without unity over $F_{\fppf}$,
define a sheaf $\alg{H}^{q}(1 + I \tensor_{W_{N}(F)} J)$ on $F_{\fppf}$
by the fppf sheafification of the presheaf
that sends an $F$-algebra $R$ to
	\begin{equation} \label{0060}
		H^{q} \bigl(
			Y,
			1 + I \tensor_{W_{N}(F)} J(R)
		\bigr),
	\end{equation}
where $I \tensor_{W_{N}(F)} J(R)$ is the tensor product of the sheaf $I$
with the constant sheaf $J(R)$.

The particular case
$\alg{H}^{q}(1 + I \tensor_{W_{N}(F)} W_{N})$ is independent of the choice of $N$:

\begin{Prop} \label{0063}
	Assume that $I$ satisfies $p^{N'} I = 0$ for $N' \le N$.
	Then the natural reduction morphism from
	$\alg{H}^{q}(1 + I \tensor_{W_{N}(F)} W_{N})$
	to $\alg{H}^{q}(1 + I \tensor_{W_{N'}(F)} W_{N'})$ is an isomorphism.
\end{Prop}

\begin{proof}
	The functor sending an $F$-algebra $R$ to
	the sheaf $1 + I \tensor_{W_{N}(F)} W_{N}(R)$ on $Y_{\et}$
	commutes with filtered direct limits.
	Since the \'etale cohomology functor $H^{q}(Y, \var)$ commutes
	with filtered direct limits (\cite[Tag 03Q5]{Sta22}),
	it follows that the sheaf $\alg{H}^{q}(1 + I \tensor_{W_{N}(F)} W_{N})$
	as a functor in $F$-algebras commutes with filtered direct limits.
	Therefore we may restrict the sheaves to the category of $F$-algebras of finite type
	for proving the proposition.
	Let $R$ be an fppf-local $F$-algebra (not necessarily of finite type),
	namely an $F$-algebra such that
	any faithfully flat $R$-algebra of finite presentation admits a retract
	(\cite[Definition 0.1]{GK15}).
	By what we saw above and \cite[Theorem 0.2]{GK15},%
	\footnote{
		In this reference, the underlying category of the site is
		the category of separated $F$-schemes of finite type.
		Separated may be replaced by affine as noted before.
	}
	it is enough to show that the morphism in question is an isomorphism on $R$-valued points.
	Since $R$ is fppf-local, the fppf sheafification is not needed on $R$-valued points, so
		\begin{gather*}
					\alg{H}^{q}(1 + I \tensor_{W_{N}(F)} W_{N})(R)
				\cong
					H^{q} \bigl(
						Y,
						1 + I \tensor_{W_{N}(F)} W_{N}(R)
					\bigr),
			\\
					\alg{H}^{q}(1 + I \tensor_{W_{N'}(F)} W_{N'})(R)
				\cong
					H^{q} \bigl(
						Y,
						1 + I \tensor_{W_{N'}(F)} W_{N'}(R)
					\bigr).
		\end{gather*}
	The natural map $W_{N'}(F) \tensor_{W_{N}(F)} W_{N}(R) \to W_{N'}(R)$
	is an isomorphism (again by $R$ being fppf-local).
	Hence
		\begin{align*}
					I \tensor_{W_{N}(F)} W_{N}(R)
			&	\cong
					I \tensor_{W_{N'}(F)} \bigl(
						W_{N'}(F) \tensor_{W_{N}(F)} W_{N}(R)
					\bigr)
			\\
			&	\cong
					I \tensor_{W_{N'}(F)} W_{N'}(R)
		\end{align*}
	over $Y_{\et}$.
\end{proof}

Note also that if $I^{2} = 0$ and $p I = p J = 0$,
then $\alg{H}^{q}(1 + I \tensor_{W_{N}(F)} J)$ is isomorphic to
$H^{q}(Y, I) \tensor_{F} J$,
where $H^{q}(Y, I)$ is viewed as a constant sheaf on $F_{\fppf}$.

Let $\Frob_{F} \colon \Ga \to \Ga$ be the relative Frobenius morphism over $F$.
Let $\alpha_{p}$ be its kernel.
For an $F$-algebra $R$, set $\Ga'(R) = W_{2}(R) / p(W_{2}(R))$
and $\alpha_{p}'(R) = W_{2}(R)[p] / p(W_{2}(R))$.
They fppf-sheafify to $\Ga$ and $\alpha_{p}$.
For an $F$-vector space $V$, let $V^{(p)} = V \tensor_{F} F$,
where the right tensor factor is the $p$-th power map $F \to F$ viewed as an $F$-algebra.
In other words, $V^{(p)}$ is $V$ with new $F$-action given by $a \cdot v = a^{1 / p} v$.
We have an exact sequence
	\begin{equation} \label{0061}
			0
		\to
			V \tensor_{F} \alpha_{p}
		\to
			V \tensor_{F} \Ga
		\stackrel{\id_{V} \tensor \Frob_{F}}{\to}
			V^{(p)} \tensor_{F} \Ga
		\to
			0
	\end{equation}
over $F_{\fppf}$.

For ordered sets of positive integers $m = (m_{1}, \dots, m_{n})$,
$m' = (m_{1}', \dots, m_{n}')$,
$k = (k_{1}, \dots, k_{l})$ and
$k' = (k_{1}', \dots, k_{l}')$ with
$m_{i} \le m_{i}'$ and $k_{j} \le k_{j}'$ for all $i$ and $j$,
we denote $I^{m, k}_{m', k'} = I_{Y}^{m} I_{Z}^{k} / I_{Y}^{m'} I_{Z}^{k'}$.
We have $(1 + I_{Y}^{m} I_{Z}^{k}) / (1+ I_{Y}^{m'} I_{Z}^{k'}) \cong 1 + I^{m, k}_{m', k'}$.
When $k = k'$, the sheaf $I^{m, k}_{m', k}$ is supported on $Y$,
so we view it as a sheaf on $Y_{\et}$.

With this language, we view the sheaf
	\[
		\alg{H}^{1} \bigl(
				1
			+
				I^{1, f_{Z}}_{f_{Y}, f_{Z}}
			\tensor_{W_{2}(F)}
				W_{2}
		\bigr)
	\]
as our algebraic structure on $H^{1}(Y, 1 + I^{1, f_{Z}}_{f_{Y}, f_{Z}})$.
To analyze it, we begin with a lemma:

\begin{Lem} \label{0007}
	Let $0 \to M_{1} \to M_{2} \to M_{3} \to 0$ be an exact sequence of $W_{2}(F)$-modules
	such that $p M_{1} = p M_{3} = 0$.
	Let $N$ be a $W_{2}(F)$-module.
	Then the kernel of the natural map
	$M_{1} \tensor_{W_{2}(F)} N \to M_{2} \tensor_{W_{2}(F)} N$
	is given by the submodule $(p M_{2}) \tensor_{F} (N[p] / p N)$.
	In particular, we have an exact sequence
		\[
				0
			\to
				(p M_{2}) \tensor_{F} (N[p] / p N)
			\to
				M_{1} \tensor_{W_{2}(F)} N
			\to
				M_{2} \tensor_{W_{2}(F)} N
			\to
				M_{3} \tensor_{W_{2}(F)} N
			\to
				0.
		\]
\end{Lem}

\begin{proof}
	Everything commutes with filtered direct limits in $N$.
	Hence we may assume that $N$ is finite over $W_{2}(F)$
	and consequently that $N$ is either $W_{2}(F)$ or $F$.
	A direct calculation in each case gives the result.
\end{proof}

Consider the exact sequence
	\[
			0
		\to
			I^{e_{Y} + 1, f_{Z}}_{f_{Y}, f_{Z}}
		\to
			I^{1, f_{Z}}_{f_{Y}, f_{Z}}
		\to
			I^{1, f_{Z}}_{e_{Y} + 1, f_{Z}}
		\to
			0.
	\]
It is an exact sequence of sheaves of $W_{2}(F)$-modules over $Y_{\et}$,
with the first and third terms killed by $p$.
For an $F$-algebra $R$,
applying Lemma \ref{0007} to this sequence and $W_{2}(R)$,
we obtain an exact sequence
	\begin{align*}
		&
				0
			\to
					I^{e_{Y} + 1, e_{Z} + f_{Z}}_{f_{Y}, e_{Z} + f_{Z}}
				\tensor_{F}
					\alpha_{p}'(R)
			\to
					I^{e_{Y} + 1, f_{Z}}_{f_{Y}, f_{Z}}
				\tensor_{F}
					\Ga'(R)
		\\
		&	\to
					I^{1, f_{Z}}_{f_{Y}, f_{Z}}
				\tensor_{W_{2}(F)}
					W_{2}(R)
			\to
					I^{1, f_{Z}}_{e_{Y} + 1, f_{Z}}
				\tensor_{F}
					\Ga'(R)
			\to
				0.
	\end{align*}
This induces an exact sequence
	\begin{equation} \label{0008}
		\begin{aligned}
			&
					0
				\to
						I^{e_{Y} + 1, e_{Z} + f_{Z}}_{f_{Y}, e_{Z} + f_{Z}}
					\tensor_{F}
						\alpha_{p}'(R)
				\to
						I^{e_{Y} + 1, f_{Z}}_{f_{Y}, f_{Z}}
					\tensor_{F}
						\Ga'(R)
			\\
			&	\to
							1
					+
						I^{1, f_{Z}}_{f_{Y}, f_{Z}}
					\tensor_{W_{2}(F)}
						W_{2}(R)
				\to
						1
					+
						I^{1, f_{Z}}_{e_{Y} + 1, f_{Z}}
					\tensor_{F}
						\Ga'(R)
				\to
					0
		\end{aligned}
	\end{equation}
of abelian sheaves on $Y_{\et}$.
Proposition \ref{0054} \eqref{0057} shows that
$- \sum_{i} (f_{Y_{i}} / p) Y_{i}$ is nef.
Hence using Proposition \ref{0054} \eqref{0055} and \eqref{0056}
with $c = p$ and $m = - f_{Y} / p$,
we have $\Gamma(Y, I^{1, f_{Z}}_{e_{Y} + 1, f_{Z}}) = 0$
and hence
	\[
			\Gamma(Y, 1 + I^{1, f_{Z}}_{e_{Y} + 1, f_{Z}} \tensor_{F} \Ga'(R))
		\cong
			1 + \Gamma(Y, I^{1, f_{Z}}_{e_{Y} + 1, f_{Z}}) \tensor_{F} \Ga'(R)
		=
			0.
	\]
Also, the cokernel of the inclusion
	$
			I^{e_{Y} + 1, e_{Z} + f_{Z}}_{f_{Y}, e_{Z} + f_{Z}}
		\into
			I^{e_{Y} + 1, f_{Z}}_{f_{Y}, f_{Z}}
	$
is a skyscraper sheaf
and hence has trivial $H^{1}$.
Therefore the sequence \eqref{0008} induces an exact sequence
	\begin{align*}
		&
				0
			\to
					H^{1}(Y, I^{e_{Y} + 1, f_{Z}}_{f_{Y}, f_{Z}})
				\tensor_{F}
					\alpha_{p}
			\to
					H^{1}(Y, I^{e_{Y} + 1, f_{Z}}_{f_{Y}, f_{Z}})
				\tensor_{F}
					\Ga
		\\
		&	\to
				\alg{H}^{1} \bigl(
						1
					+
						I^{1, f_{Z}}_{f_{Y}, f_{Z}}
					\tensor_{W_{2}(F)}
						W_{2}
				\bigr)
			\to
				\alg{H}^{1} \bigl(
						1
					+
						I^{1, f_{Z}}_{e_{Y} + 1, f_{Z}}
					\tensor_{F}
						\Ga
				\bigr)
			\to
				0
	\end{align*}
over $F_{\fppf}$.
Using the exact sequence \eqref{0061}
with $V = H^{1}(Y, I^{e_{Y} + 1, f_{Z}}_{f_{Y}, f_{Z}})$,
we thus get an exact sequence
	\[
			0
		\to
				H^{1}(Y, I^{e_{Y} + 1, f_{Z}}_{f_{Y}, f_{Z}})^{(p)}
			\tensor_{F}
				\Ga
		\to
			\alg{H}^{1} \bigl(
					1
				+
					I^{1, f_{Z}}_{f_{Y}, f_{Z}}
				\tensor_{W_{2}(F)}
					W_{2}
			\bigr)
		\to
			\alg{H}^{1} \bigl(
					1
				+
					I^{1, f_{Z}}_{e_{Y} + 1, f_{Z}}
				\tensor_{F}
					\Ga
			\bigr)
		\to
			0.
	\]
The third term
	$
		\alg{H}^{1} \bigl(
				1
			+
				I^{1, f_{Z}}_{e_{Y} + 1, f_{Z}}
			\tensor_{F}
				\Ga
		\bigr)
	$
is represented by a unipotent algebraic group scheme over $F$ with
	\[
			\alg{H}^{1} \bigl(
					1
				+
					I^{1, f_{Z}}_{e_{Y} + 1, f_{Z}}
				\tensor_{F}
					\Ga
			\bigr)(R)
		\cong
			H^{1}(Y, 1 + I^{1, f_{Z}}_{e_{Y} + 1, f_{Z}} \tensor_{F} R)
	\]
for all $F$-algebras $R$
(use the filtrations in Proposition \ref{0054} \eqref{0055} and \eqref{0056}
to reduce the statement to what is explained right after Proposition \ref{0063}).
In particular, we obtain:

\begin{Prop}
	The sheaf
		$
			\alg{H}^{1} \bigl(
					1
				+
					I^{1, f_{Z}}_{f_{Y}, f_{Z}}
				\tensor_{W_{2}(F)}
					W_{2}
			\bigr)
		$
	is represented by a unipotent algebraic group scheme over $F$.
\end{Prop}

Taking the Lie algebras, we obtain an exact sequence
	\begin{equation} \label{0009}
			0
		\to
			H^{1}(Y, I^{e_{Y} + 1, f_{Z}}_{f_{Y}, f_{Z}})^{(p)}
		\to
			\Lie
			\alg{H}^{1} \bigl(
					1
				+
					I^{1, f_{Z}}_{f_{Y}, f_{Z}}
				\tensor_{W_{2}(F)}
					W_{2}
			\bigr)
		\to
			H^{1}(Y, I^{1, f_{Z}}_{e_{Y} + 1, f_{Z}})
		\to
			0
	\end{equation}
of $F$-vector spaces.
This sequence canonically splits:
Consider the natural map
	\[
			H^{1} \bigl(
				Y,
					1
				+
					I^{1, f_{Z}}_{f_{Y}, f_{Z}}
				\tensor_{W_{2}(F)}
					W_{2}(F[\varepsilon])
			\bigr)
		\to
			\alg{H}^{1} \bigl(
					1
				+
					I^{1, f_{Z}}_{f_{Y}, f_{Z}}
				\tensor_{W_{2}(F)}
					W_{2}
			\bigr)(F[\varepsilon]),
	\]
where $F[\varepsilon] \cong F[x] / (x^{2})$.
Since $W_{2}(F[\varepsilon]) \cong W_{2}(F) \oplus F (\varepsilon, 0) \oplus F (0, \varepsilon)$
as $W_{2}(F)$-modules,
the middle summand $F (\varepsilon, 0)$ gives an $F$-linear map
	\[
			H^{1}(Y, I^{1, f_{Z}}_{f_{Y}, f_{Z}} \tensor_{W_{2}(F)} F)
		\to
			\Lie
			\alg{H}^{1} \bigl(
					1
				+
					I^{1, f_{Z}}_{f_{Y}, f_{Z}}
				\tensor_{W_{2}(F)}
					W_{2}
			\bigr).
	\]
The left-hand side is isomorphic to
$H^{1}(Y, I^{1, f_{Z}}_{e_{Y} + 1, f_{Z}})$,
since the image of multiplication by $p$ on $I^{1, f_{Z}}_{f_{Y}, f_{Z}}$ is
$I^{e_{Y} + 1, e_{Z} + f_{Z}}_{f_{Y}, e_{Z} + f_{Z}}$,
which injects into $I^{e_{Y} + 1, f_{Z}}_{f_{Y}, f_{Z}}$
with skyscraper cokernel as noted before.
The resulting $F$-linear map from $H^{1}(Y, I^{1, f_{Z}}_{e_{Y} + 1, f_{Z}})$
gives the desired splitting.

We describe the first map in \eqref{0009}.
First note that all the sheaves in \eqref{0008} have
trivial cohomology in positive degrees over any affine scheme \'etale over $Y$.
Also, $Y$ can be covered by two affine opens.
We use the following to describe $H^{1}$:

\begin{Lem} \label{0019}
	Let $Y = U \cup V$ be an affine open cover.
	Let $G \in \Ab(Y_{\et})$ be a sheaf such that
	$H^{1}(U, G) = H^{1}(V, G) = 0$.
	Then we have an exact sequence
		\[
				0
			\to
				G(Y)
			\to
				G(U) \oplus G(V)
			\to
				G(U \cap V)
			\to
				H^{1}(Y, G)
			\to
				0.
		\]
\end{Lem}

\begin{proof}
	Obvious.
\end{proof}

Now the first map in \eqref{0009} is described as follows.
We work over $F[\varepsilon'] \cong F[x] / (x^{2 p})$.
It is an fppf cover of $F[\varepsilon]$
via the $F$-algebra map $\varepsilon \mapsto \varepsilon'^{p}$.
In the situation of Lemma \ref{0019} (so $Y = U \cup V$ is an affine open cover),
let $\beta$ be an element of $\Gamma(U \cap V, I^{e_{Y} + 1, f_{Z}}_{f_{Y}, f_{Z}})^{(p)}$.
Then $\beta \tensor \varepsilon$ is an element of
	\[
			\Gamma(U \cap V, I^{e_{Y} + 1, f_{Z}}_{f_{Y}, f_{Z}} \tensor_{F} F[\varepsilon])^{(p)}
		\subset
			\Gamma(U \cap V, I^{e_{Y} + 1, f_{Z}}_{f_{Y}, f_{Z}} \tensor_{F} F[\varepsilon'])^{(p)}.
	\]
Take any element $\alpha$ of
$\Gamma(U \cap V, I^{e_{Y} + 1, f_{Z}}_{f_{Y}, f_{Z}} \tensor_{F} F[\varepsilon'])$
that maps to $\beta \tensor \varepsilon$ via $\id \tensor \Frob_{F}$
(for example, $\alpha = \beta \tensor \varepsilon'$).
Any lift of $1 + \alpha$ as a section of
	$
			1
		+		I^{1, f_{Z}}_{f_{Y}, f_{Z}}
			\tensor_{W_{2}(F)}
				W_{2}(F[\varepsilon'])
	$
over $U \cap V$ defines an element $\gamma$ of
	\[
			\alg{H}^{1} \bigl(
					1
				+
					I^{1, f_{Z}}_{f_{Y}, f_{Z}}
				\tensor_{W_{2}(F)}
					W_{2}
			\bigr)(F[\varepsilon])
		\subset
			\alg{H}^{1} \bigl(
					1
				+
					I^{1, f_{Z}}_{f_{Y}, f_{Z}}
				\tensor_{W_{2}(F)}
					W_{2}
			\bigr)(F[\varepsilon'])
	\]
that maps to zero via $\varepsilon \mapsto 0$.
Hence $\gamma$ is an element of the Lie algebra of
	$
		\alg{H}^{1} \bigl(
				1
			+
				I^{1, f_{Z}}_{f_{Y}, f_{Z}}
			\tensor_{W_{2}(F)}
				W_{2}
		\bigr)
	$.
Now the map in question assigns $\gamma$ to $\beta$.

The upshot is:
\begin{Prop} \label{0013}
	We have a canonical isomorphism
		\[
				\Lie
				\alg{H}^{1} \bigl(
						1
					+
						I^{1, f_{Z}}_{f_{Y}, f_{Z}}
					\tensor_{W_{2}(F)}
						W_{2}
				\bigr)
			\cong
					H^{1}(Y, I^{e_{Y} + 1, f_{Z}}_{f_{Y}, f_{Z}})^{(p)}
				\oplus
					H^{1}(Y, I^{1, f_{Z}}_{e_{Y} + 1, f_{Z}})
		\]
	with the maps described as above.
\end{Prop}

%%%%%%%%%%%%%%%%%%%%%%%%%%%%%%%%%%%%%%%%%%%%%%%%%%%%%%%%%%%%%%%%%%%%%%%%%%%%%

\section{Lie algebras of deformation cohomology II}
\label{0037}

Now we treat the group
$H^{1} \bigl( \mathfrak{X}, (1 + I_{Y} I_{Z}^{f_{Z} / p}) / (1 + I_{Y}^{f_{Y} / p} I_{Z}^{f_{Z} / p}) \bigr)$
in Proposition \ref{0016}.
Our algebraic group structure is different from the obvious candidate
in that certain infinitesimals are killed.

Define an fppf sheaf
	\begin{equation} \label{0011}
		\alg{H}^{1} \left(
			\frac{
				1 + I^{1, f_{Z} / p}_{f_{Y} / p, f_{Z} / p} \tensor_{F} \Ga
			}{
				1 + I^{1, f_{Z}}_{f_{Y} / p, f_{Z}} \tensor_{F} \alpha_{p}
			}
		\right)
	\end{equation}
on $F$ to be the fppf sheafification of the presheaf
	\[
			R
		\mapsto
		H^{1} \left(
			Y,\,
			\frac{
				1 + I^{1, f_{Z} / p}_{f_{Y} / p, f_{Z} / p} \tensor_{F} R
			}{
				1 + I^{1, f_{Z}}_{f_{Y} / p, f_{Z}} \tensor_{F} \alpha_{p}(R)
			}
		\right).
	\]
Using Proposition \ref{0054} \eqref{0055}
with $c = 1$ and $m = 0$,
we have $\Gamma(Y, I^{1, f_{Z} / p}_{f_{Y} / p, f_{Z} / p}) = 0$.
Also, the cokernel of the inclusion
$I^{1, f_{Z} / p}_{f_{Y} / p, f_{Z} / p} \into I^{1, f_{Z}}_{f_{Y} / p, f_{Z}}$
is a skyscraper sheaf and has trivial $H^{1}$.
Hence we have an exact sequence
	\[
			0
		\to
			\alg{H}^{1} \bigl(
				1 + I^{1, f_{Z} / p}_{f_{Y} / p, f_{Z} / p} \tensor_{F} \alpha_{p}
			\bigr)
		\to
			\alg{H}^{1}(
				1 + I^{1, f_{Z} / p}_{f_{Y} / p, f_{Z} / p} \tensor_{F} \Ga
			)
		\to
			\alg{H}^{1} \left(
				\frac{
					1 + I^{1, f_{Z} / p}_{f_{Y} / p, f_{Z} / p} \tensor_{F} \Ga
				}{
					1 + I^{1, f_{Z}}_{f_{Y} / p, f_{Z}} \tensor_{F} \alpha_{p}
				}
			\right)
		\to
			0.
	\]
Since the first two terms are unipotent algebraic group schemes over $F$,
so is the third term.
For any $F$-algebra $R$,
we have an exact sequence
	\[
			0
		\to
			I^{1, f_{Z} / p}_{f_{Y} / p, f_{Z} / p} \tensor_{F} \alpha_{p}(R)
		\to
			I^{1, f_{Z} / p}_{f_{Y} / p, f_{Z} / p} \tensor_{F} R
		\stackrel{\id \tensor \Frob_{F}}{\to}
			(I^{1, f_{Z} / p}_{f_{Y} / p, f_{Z} / p})^{(p)} \tensor_{F} R^{p}
		\to
			0
	\]
over $Y_{\et}$,
where $(I^{1, f_{Z} / p}_{f_{Y} / p, f_{Z} / p})^{(p)}$ is
the sheaf $I^{1, f_{Z} / p}_{f_{Y} / p, f_{Z} / p}$
with $F$-action given by $a \cdot v = a^{1 / p} v$
and $R^{p}$ is the set of $p$-th powers in $R$.
Therefore \eqref{0011} is isomorphic to
	\[
		\alg{H}^{1} \bigl(
			1 + (I^{1, f_{Z} / p}_{f_{Y} / p, f_{Z} / p})^{(p)} \tensor_{F} \Ga
		\bigr).
	\]
Its Lie algebra is
	\[
			H^{1}(Y, (I^{1, f_{Z} / p}_{f_{Y} / p, f_{Z} / p})^{(p)})
		\cong
			H^{1}(Y, I^{1, f_{Z} / p}_{f_{Y} / p, f_{Z} / p})^{(p)}.
	\]
Hence by taking the Lie algebras, we have
	\begin{equation} \label{0012}
			\Lie
			\alg{H}^{1} \left(
				\frac{
					1 + I^{1, f_{Z} / p}_{f_{Y} / p, f_{Z} / p} \tensor_{F} \Ga
				}{
					1 + I^{1, f_{Z}}_{f_{Y} / p, f_{Z}} \tensor_{F} \alpha_{p}
				}
			\right)
		\cong
			H^{1}(Y, I^{1, f_{Z} / p}_{f_{Y} / p, f_{Z} / p})^{(p)}.
	\end{equation}

This isomorphism is described as follows.
In the situation of Lemma \ref{0019},
let $\beta$ be an element of
	\[
			\Gamma(U \cap V, I^{1, f_{Z} / p}_{f_{Y} / p, f_{Z} / p})^{(p)}
		\cong
			\Gamma \Bigl(
				U \cap V,
				(I^{1, f_{Z} / p}_{f_{Y} / p, f_{Z} / p})^{(p)}
			\Bigr).
	\]
Then $\beta \tensor \varepsilon$ is an element of
	\[
			\Gamma \Bigl(
				U \cap V,
					(I^{1, f_{Z} / p}_{f_{Y} / p, f_{Z} / p})^{(p)}
				\tensor_{F}
					F[\varepsilon]
			\Bigr)
		\subset
			\Gamma \Bigl(
				U \cap V,
					(I^{1, f_{Z} / p}_{f_{Y} / p, f_{Z} / p})^{(p)}
				\tensor_{F}
					F[\varepsilon']
			\Bigr).
	\]
Take any element $\alpha$ of
$\Gamma(U \cap V, I^{1, f_{Z} / p}_{f_{Y} / p, f_{Z} / p} \tensor_{F} F[\varepsilon'])$
that maps to $\beta \tensor \varepsilon$ via $\id \tensor \Frob_{F}$
(for example, $\alpha = \beta \tensor \varepsilon'$).
The element $1 + \alpha$ defines an element $\gamma$ of
	\[
			\alg{H}^{1} \left(
				\frac{
					1 + I^{1, f_{Z} / p}_{f_{Y} / p, f_{Z} / p} \tensor_{F} \Ga
				}{
					1 + I^{1, f_{Z}}_{f_{Y} / p, f_{Z}} \tensor_{F} \alpha_{p}
				}
			\right)(F[\varepsilon])
		\subset
			\alg{H}^{1} \left(
				\frac{
					1 + I^{1, f_{Z} / p}_{f_{Y} / p, f_{Z} / p} \tensor_{F} \Ga
				}{
					1 + I^{1, f_{Z}}_{f_{Y} / p, f_{Z}} \tensor_{F} \alpha_{p}
				}
			\right)(F[\varepsilon'])
	\]
that maps to zero via $\varepsilon \mapsto 0$.
Hence $\gamma$ is an element of the Lie algebra of \eqref{0011}.
Now the isomorphism \eqref{0012} assigns $\gamma$ to $\beta$.

We compare the two algebraic group schemes thus obtained.

\begin{Prop}
	For any $F$-algebra $R$,
	the $p$-th power map (endomorphism) on the sheaf
	$1 + I^{1, f_{Z}}_{f_{Y}, f_{Z}} \tensor_{W_{2}(F)} W_{2}(R)$
	on $Y_{\et}$ factors through the quotient
	$1 + I^{1, f_{Z} / p}_{f_{Y} / p, f_{Z} / p} \tensor_{F} \Ga'(R)$
	of $1 + I^{1, f_{Z}}_{f_{Y}, f_{Z}} \tensor_{W_{2}(F)} W_{2}(R)$,
	defining a morphism
		\begin{equation} \label{0062}
				p
			\colon
				1 + I^{1, f_{Z} / p}_{f_{Y} / p, f_{Z} / p} \tensor_{F} \Ga'(R)
			\to
				1 + I^{1, f_{Z}}_{f_{Y}, f_{Z}} \tensor_{W_{2}(F)} W_{2}(R).
		\end{equation}
	
	We have a commutative diagram
		\[
			\begin{CD}
					1 + I^{1, f_{Z}}_{f_{Y} / p, f_{Z}} \tensor_{F} \alpha_{p}'(R)
				@>>>
					1 + I^{1, f_{Z} / p}_{f_{Y} / p, f_{Z} / p} \tensor_{F} \Ga'(R)
				\\
				@VVV @VV p V
				\\
					I^{e_{Y} + 1, f_{Z}}_{f_{Y}, f_{Z}} \tensor_{F} \Ga'(R)
				@>>>
					1 + I^{1, f_{Z}}_{f_{Y}, f_{Z}} \tensor_{W_{2}(F)} W_{2}(R)
			\end{CD}
		\]
	of sheaves on $Y_{\et}$,
	where the horizontal morphisms are the natural ones and
	the left vertical morphism sends a section $1 + \alpha$ to $(1 + \alpha)^{p} - 1$.
	
	Moreover, the image of the left vertical morphism is contained in the subsheaf
	$I^{e_{Y} + 1, e_{Z} + f_{Z}}_{f_{Y}, e_{Z} + f_{Z}} \tensor_{F} \alpha_{p}'(R)$.
\end{Prop}

\begin{proof}
	For the well-definedness of \eqref{0062},
	it is enough to see that the $p$-th power map kills the images of the sheaves
	$1 + I^{f_{Y} / p, f_{Z} / p}_{f_{Y}, f_{Z}} \tensor_{W_{2}(F)} W_{2}(R)$
	and
	$1 + I^{1, f_{Z}}_{f_{Y}, f_{Z}} \tensor_{W_{2}(F)} p (W_{2}(R))$.
	For the rest,
	the only thing to note is that $\alpha^{p}$ is zero
	and $p \alpha^{i}$ for $1 \le i \le p - 1$ is a section of
	$I^{e_{Y} + 1, e_{Z} + f_{Z}}_{f_{Y}, e_{Z} + f_{Z}} \tensor_{F} \alpha_{p}'(R)$.
\end{proof}

Therefore we obtain a well-defined morphism
	\begin{equation} \label{0015}
			\alg{H}^{1} \left(
				\frac{
					1 + I^{1, f_{Z} / p}_{f_{Y} / p, f_{Z} / p} \tensor_{F} \Ga
				}{
					1 + I^{1, f_{Z}}_{f_{Y} / p, f_{Z}} \tensor_{F} \alpha_{p}
				}
			\right)
		\stackrel{p}{\to}
			\alg{H}^{1} \bigl(
					1
				+
					I^{1, f_{Z}}_{f_{Y}, f_{Z}}
				\tensor_{W_{2}(F)}
					W_{2}
			\bigr)
	\end{equation}
of unipotent algebraic group schemes over $F$.
By Proposition \ref{0013} and Equation \eqref{0012},
this induces an $F$-linear map
	\begin{equation} \label{0014}
			H^{1}(Y, I^{1, f_{Z} / p}_{f_{Y} / p, f_{Z} / p})^{(p)}
		\to
				H^{1}(Y, I^{e_{Y} + 1, f_{Z}}_{f_{Y}, f_{Z}})^{(p)}
			\oplus
				H^{1}(Y, I^{1, f_{Z}}_{e_{Y} + 1, f_{Z}})
	\end{equation}
on the Lie algebras.

On the other hand, the multiplication by $p$ and the $p$-th power map
induce an isomorphism and a morphism
	\[
			p
		\colon
			I^{1, f_{Z} / p}_{f_{Y} / p, f_{Z} / p}
		\isomto
			I^{e_{Y} + 1, f_{Z}}_{f_{Y}, f_{Z}}
		\quad \text{and} \quad
			\Frob
		\colon
			I^{1, f_{Z} / p}_{f_{Y} / p, f_{Z} / p}
		\to
			I^{1, f_{Z}}_{e_{Y} + 1, f_{Z}},
	\]
respectively, over $Y_{\et}$.
The maps induced on $H^{1}$ are also denoted by $p$ and $\Frob$.

The following key result is false
if we do not factor out by $1 + I^{1, f_{Z}}_{f_{Y} / p, f_{Z}} \tensor_{F} \alpha_{p}$
in the definition \eqref{0011}.

\begin{Prop}
	The map \eqref{0014} is given by $(p, \Frob)$.
	It is injective.
\end{Prop}

\begin{proof}
	In the situation of Lemma \ref{0019},
	let $\beta$ be an element of
	$\Gamma(U \cap V, I^{1, f_{Z} / p}_{f_{Y} / p, f_{Z} / p})^{(p)}$.
	Then $\alpha = \beta \tensor \varepsilon'$ is a section of
	$I^{1, f_{Z} / p}_{f_{Y} / p, f_{Z} / p} \tensor_{F} F[\varepsilon']$
	over $U \cap V$ that maps to $\beta \tensor \varepsilon$ via $\id \tensor \Frob_{F}$.
	The element $(1 + \beta \tensor (\varepsilon', 0))^{p}$ gives a section of
		$
				1
			+		I^{1, f_{Z}}_{f_{Y}, f_{Z}}
				\tensor_{W_{2}(F)}
					W_{2}(F[\varepsilon'])
		$.
	Its image in
		$
				1
			+
				I^{1, f_{Z}}_{e_{Y} + 1, f_{Z}}
			\tensor_{F}
				\Ga
		$
	is $1 + \beta^{p} \tensor \varepsilon$.
	Thus the second component of \eqref{0014} is $\Frob$.
	We have
		\[
				\frac{
					(1 + \beta \tensor (\varepsilon', 0))^{p}
				}{
					1 + \beta^{p} \tensor (\varepsilon'^{p}, 0)
				}
			=
					1
				+
					\frac{
						\sum_{i = 1}^{p - 1}
							\binom{p}{i}
							\beta^{i} \tensor (\varepsilon'^{i}, 0)
					}{
						1 + \beta^{p} \tensor (\varepsilon'^{p}, 0)
					},
		\]
	which is a lift of the section
		\[
			\frac{
				\sum_{i = 1}^{p - 1}
					\binom{p}{i}
					\beta^{i} \tensor \varepsilon'^{i}
			}{
				1 + \beta^{p} \tensor \varepsilon'^{p}
			}
		\]
	of $I^{e_{Y} + 1, f_{Z}}_{f_{Y}, f_{Z}} \tensor_{F} F[\varepsilon']$.
	Its image by $\id \tensor \Frob_{F}$ is $p \beta \tensor \varepsilon$.
	Hence the first component of \eqref{0014} is $p$.
	This first component is an isomorphism.
	Hence \eqref{0014} is injective.
\end{proof}

Therefore the morphism \eqref{0015} has finite \'etale kernel.
Taking $F$-valued points, we know that the map
	\[
			H^{1}(Y, 1 + I^{1, f_{Z} / p}_{f_{Y} / p, f_{Z} / p})
		\stackrel{p}{\to}
			H^{1}(Y, 1 + I^{1, f_{Z}}_{f_{Y}, f_{Z}})
	\]
has finite kernel.
By Proposition \ref{0016},
this implies that
	$
		\Gamma \bigl(
			Y,
			U^{(1, f_{Z})} R^{1} \Psi \mathfrak{T}(1)
		\bigr)
	$
is finite.
This finishes the proof of Theorem \ref{0018} in the mixed characteristic case.

%%%%%%%%%%%%%%%%%%%%%%%%%%%%%%%%%%%%%%%%%%%%%%%%%%%%%%%%%%%%%%%%%%%%%%%%%%%%%

\end{document}